\documentclass[amssymb,twoside,12pt]{article}
\usepackage{latexsym,amsmath,graphicx,mathrsfs,amssymb}
\usepackage{amsfonts}
\usepackage{enumitem}
\newtheorem{theorem}{Theorem}[section]
\newtheorem{lemma}[theorem]{{\bf Lemma}}

\newtheorem{rem}[theorem]{{\bf Remark}}
\newtheorem{ex}[theorem]{{\bf Example}}
\newtheorem{definition}{Definition}[section]
\numberwithin{equation}{section}
\newenvironment{proof}{\indent{\em Proof:}}{\quad \hfill
$\Box$\vspace*{2ex}}

\begin{document}
\setcounter{page}{1}
\begin{center}
\vspace{0.3cm} {\large{\bf Analysis of Volterra Integrodifferential Equations with Nonlocal and Boundary Conditions via  Picard Operator}} \\
\vspace{0.4cm}
 Pallavi U. Shikhare $^{1}$ \\
jananishikhare13@gmail.com  \\

\vspace{0.30cm}
Kishor D. Kucche $^{2}$\\
kdkucche@gmail.com\\

\vspace{0.30cm}
 J. Vanterler da C. Sousa  $^{3}$\\
ra160908@ime.unicamp.br\\
\vspace{0.35cm}

\vspace{0.30cm}
$^{1,2}$ Department of Mathematics, Shivaji University, Kolhapur-416 004, Maharashtra, India.\\
$^{3}$  Department of Applied Mathematics, Imecc-Unicamp, 13083-859, Campinas, SP, Brazil.
\end{center}

\def\baselinestretch{1.0}\small\normalsize
\begin{abstract}
This article investigates the existence and uniqueness of solutions to the second order Volterra integrodifferential equations with nonlocal and boundary conditions through  its integral equivalent equations and fixed point of Banach.  Further, utilising the Picard operator theory we obtain the dependency of solutions on the initial nonlocal data and on functions involved on the right hand side of the equations. 
\end{abstract}
\allowdisplaybreaks
\section{Introduction}

Over the years, the study of solutions of differential equations has been the subject of research, and continues for a number of reasons, from theoretical comfort, which is about existence, uniqueness, controllability, among others, and in the practical sense, involving the stability and continuous dependence on data \cite{Byszewski1,Byszewski3,Bednarz1,Byszewski4,bala,Bednarz}. So there are several attractive points to investigate the solutions of the various types of differential equations. On the other hand, we can also highlight the integrodifferential equations that have gained prominence in the scientific community \cite{bala,Lin,Muresan, Muresan1, Kucche3,Kucche4,Kucche5,kucche4,Kucche2}.

The nonlocal condition is a generalization of the classical initial condition.   Studies with the nonlocal conditions is driven by theoretical premium, yet additionally there are several events happened in   engineering, physics and life sciences that can be described by means of differential equations subject to nonlocal conditions \cite{Bates}--\cite{ Delgado1}. 
Therefore differential equations with nonlocal condition has  turned into an active  zone of research . 

 In 1991, Byszewski \cite{Byszewski1} introduced the nonlocal Cauchy problem in abstract spaces. The following years, Byszewski \cite{Byszewski3}, carried out another important work, also addressing the existence and uniqueness of classical solutions to a nonlocal Cauchy problem of the abstract type in a Banach space.
 In the literature many researcher has been commented on nonlocal conditions and investigated various class of differential and integrodifferential equations for existences, uniqueness and dependency of solutions \cite{Byszewski2,Balachandran1,Balachandran2,Teresa}.

Wang et al. \cite{Wang}, utilizing the Picard and weakly Picard  operator theory combined with   Bielecki  norms analysed
the nonlocal Cauchy problem in Banach spaces of the form: 
$$
w'(t)=f\left( t, w(t)\right),\, t\in [0,b], ~ w(0)= w_{0}+g(w),
$$ 
for   existence, uniqueness  and dependency of solutions. On the other hand, Otrocol and Ilea \cite{Ilea,Otrocol}, using the technique of Picard weak operators, examined the existence and qualitative properties of solutions for differential and integrodifferential  equations with abstract Volterra operators. 

There are many other important and interesting works related to this theme that will contribute to the growth of the theory and make possible new research, so we suggest some work for a more in-depth reading \cite{Byszewski4,Muresan,Teresa,Byszewski5,Egri,Kucche0}.  
 
Motivated by \cite{Byszewski3,Wang,Ilea,Byszewski5},  the main objective of this paper is to discuss some basic problems, such as, existence and uniqueness and dependency of solutions of the following second order Volterra integrodifferential equations with nonlocal and boundary conditions (VIDNBC)
\begin{small}
\begin{align}
& w^{''}(t)=\mathscr F\left( t,  w(t),w'(t),\int_{0}^{t}\mathscr G\left( t, s, w(s),w'(s)\right) ds\right),\,t\in J=[0,T],\, T>0,\label{1}\\
& w(0)+\sum_{k=1}^{p} c_{k} w(t_{k})=w_{0},\label{2}\\
& w'(T)= \beta\, w'(0),\,\, 1<\beta<+\infty, \label{3}
\end{align}
\end{small}
by utilizing the fixed point theorem and Picard operator theory,
where $0<t_{1}<t_{2}<\cdots<t_{p}\leq T$; $\mathscr F:J \times \mathbb{R} \times\mathbb{R}\times\mathbb{R} \to \mathbb{R} ,\, \mathscr G: J \times J \times\mathbb{R}\times\mathbb{R} \to \mathbb{R}$ are given functions satisfying some assumptions that will be specified later and $c_{1},c_{2}\cdots c_{p},\, (p\in \mathbb{N})$ are the constants such that $\sum\limits_{k=1}^{p} c_{k} \not = -1$.  

The growth and branching of the field of differential and integrodifferential equations with nonlocal and boundary conditions provide the scientific and the academic community the main reasons for the elaboration of the present paper. A systematic and detailed study under taken on the existence, uniqueness and dependencies of the data, provides a wider range of tools and work that may prove to be useful and important for future research.

This paper is organized as follows. Find out preliminarily facts in section 2. In section 3, we will investigate the first main result of the article, that is, the existence and uniqueness of solutions of   second order Volterra integrodifferential equations with nonlocal and boundary conditions, through Banach's contraction principle. Sections 4, identified with dependency of solutions through Picard operators theory. In section 5, we illustrate an example to verify the results. Concluding remakes closed the paper.

\section{Preliminaries} \label{preliminaries}
\begin{definition} [\cite{Wang},\cite{Rus1},\cite{Rus4}]. 
Let $\left(X, d \right)$  be a metric space. An operator $\mathcal{A}:X \rightarrow X $ is a Picard operator (PO), if there exists $w^{\ast }\in X $ satisfying the following conditions:
\begin{flushleft}
\begin{tabular}{cl}
{\rm (a)} & $\mathbf{F}_{\mathcal{A}}=\{w^{\ast }\}$, where $\mathbf{F}_{\mathcal{A}}:=\{w\in X  : \mathcal{A}\left( w\right) =w\}$. \\ 
{\rm (b)} & the sequence $\left( \mathcal{A}^{n}(w_{0})\right) _{n\in \mathbb{N}}$ converges to $w^{\ast }$ for all $w_{0}\in X $.
\end{tabular}
\end{flushleft}
\end{definition}
\begin{theorem}[\cite{Wang},\cite{Rus1},\cite{Rus4}]. \label{B1}
Let $(Y,d)$ be a complete metric space and $\mathcal{A},\mathcal{B}: Y\to Y$ two operators. We suppose the following:
\begin{flushleft}
\begin{tabular}{cl}
{\rm (a)} & $\mathcal{A}$ is a contraction with contraction constant $\alpha$ and $\mathbf{F}_{\mathcal{A}}=\{w^{\ast }_{\mathcal{A}}\}$ ; \\ 
{\rm (b)} &  $\mathcal{B}$ has fixed point and $w^{\ast }_{\mathcal{B}}\in\mathbf{F}_{\mathcal{B}} $;\\
{\rm (c)} & there exist $\rho>0$ such that $d\left( \mathcal{A}(w),\mathcal{B}(w)\right)\leq\rho$ for all $w\in Y.$
\end{tabular}
\end{flushleft}
Then $$d \left( w^{\ast }_{\mathcal{A}},w^{\ast }_{\mathcal{B}}\right)  \leq \frac{\rho}{1-\alpha}.$$
\end{theorem}

\section{Existence and uniqueness:}
\begin{definition}
The function $w\in C^{2}(J,\mathbb{R})$ is said to be solution of the problem \eqref{1}--\eqref{3} if $w$ satisfied  \eqref{1} and nonlocal and boundary conditions  \eqref{2} and  \eqref{3} respectively.
\end{definition}
\begin{lemma}
A function $w\in C^{2}(J,\mathbb{R})$ is a solution of VIDNBC \eqref{1}--\eqref{3} if and only if $w\in C^{1}(J,\mathbb{R})$ is a solution of the following integral equations:
\begin{small}
\begin{align}\label{e1}
& w(t)=\left( w_{0}-\sum_{k=1}^{p}c_{k}\left[ \frac{t_{k}}{\beta-1} \int_{0}^{T}\mathscr F\left( s,  w(s),w'(s),\int_{0}^{s}\mathscr G\left( s, \sigma, w(\sigma),w'(\sigma)\right)d\sigma\right) ds \right.\right.\nonumber\\
&\qquad\left.\left.+\int_{0}^{t_{k}}(t_{k}-s)\mathscr F\left( s,  w(s),w'(s),\int_{0}^{s}\mathscr G\left( s, \sigma, w(\sigma),w'(\sigma)\right)d\sigma\right) ds\right] \right)\bigg/\left( 1+\sum_{k=1}^{p}c_{k}\right) \nonumber\\
&\qquad \quad +\frac{t}{\beta-1} \int_{0}^{T}\mathscr F\left( s,  w(s),w'(s),\int_{0}^{s}\mathscr G\left( s, \sigma, w(\sigma),w'(\sigma)\right)d\sigma\right) ds \nonumber\\
&\qquad \qquad+\int_{0}^{t}(t-s)\mathscr F\left( s,  w(s),w'(s),\int_{0}^{s}\mathscr G\left( s, \sigma, w(\sigma),w'(\sigma)\right)d\sigma\right) ds.
\end{align}
\end{small}
\end{lemma}
\begin{proof}
Let $w\in C^{2}(J,\mathbb{R})$ is a solution of problem \eqref{1}--\eqref{3}.
Integrating  \eqref{1} from $0$ from $t$, we get
\begin{small}
\begin{align}\label{e2}
w'(t)=w'(0)+\int_{0}^{t}\mathscr F\left( s,  w(s),w'(s),\int_{0}^{s}\mathscr G\left( s, \sigma, w(\sigma),w'(\sigma)\right)d\sigma\right) ds. 
\end{align}
\end{small}
Again integrate above equation from $0$ from $t$, we have
\begin{small}
\begin{align}\label{e3}
w(t)=w(0)+w'(0)t+\int_{0}^{t}\mathscr (t-s) \mathscr F\left( s,  w(s),w'(s),\int_{0}^{s}\mathscr G\left( s, \sigma, w(\sigma),w'(\sigma)\right)d\sigma\right) ds. 
\end{align}
\end{small}
From \eqref{e2}, we have
\begin{small}
\begin{align}\label{e3.1}
w'(T)=w'(0)+\int_{0}^{T}\mathscr F\left( s,  w(s),w'(s),\int_{0}^{s}\mathscr G\left( s, \sigma, w(\sigma),w'(\sigma)\right)d\sigma\right) ds.
\end{align}
\end{small}
From \eqref{e3.1} and \eqref{3}, we have 
\begin{small}
\begin{align*}
 \beta\, w'(0)=w'(0)+\int_{0}^{T}\mathscr F\left( s,  w(s),w'(s),\int_{0}^{s}\mathscr G\left( s, \sigma, w(\sigma),w'(\sigma)\right)d\sigma\right) ds.
\end{align*}
\end{small}
This gives
\begin{small}
\begin{align}\label{e4}
 w'(0)=\frac{1}{\beta-1}\int_{0}^{T}\mathscr F\left( s,  w(s),w'(s),\int_{0}^{s}\mathscr G\left( s, \sigma, w(\sigma),w'(\sigma)\right)d\sigma\right) ds.
\end{align}
\end{small}
Putting the value of $w'(0)$ from \eqref{e4} in  equation \eqref{e3},  we obtain 
\begin{small}
\begin{align}\label{e5}
w(t)& =w(0)+\frac{t}{\beta-1}\int_{0}^{T}\mathscr F\left( s,  w(s),w'(s),\int_{0}^{s}\mathscr G\left( s, \sigma, w(\sigma),w'(\sigma)\right)d\sigma\right) ds\nonumber\\
&\qquad +\int_{0}^{t}\mathscr (t-s) \mathscr F\left( s,  w(s),w'(s),\int_{0}^{s}\mathscr G\left( s, \sigma, w(\sigma),w'(\sigma)\right)d\sigma\right) ds.
\end{align}
\end{small}
Therefore for any $k\, (k= 1,2,\dots,p)$
\begin{small}
\begin{align}\label{e3.2}
w(t_{k})& =w(0)+\frac{t_{k}}{\beta-1}\int_{0}^{T}\mathscr F\left( s,  w(s),w'(s),\int_{0}^{s}\mathscr G\left( s, \sigma, w(\sigma),w'(\sigma)\right)d\sigma\right) ds\nonumber\\
&\qquad +\int_{0}^{t_{k}}\mathscr (t_{k}-s) \mathscr F\left( s,  w(s),w'(s),\int_{0}^{s}\mathscr G\left( s, \sigma, w(\sigma),w'(\sigma)\right)d\sigma\right) ds.
\end{align}
\end{small}
Using \eqref{e3.2} in the nonlocal conditions \eqref{2} we obtain  
\begin{small}
\begin{align*}
& w(0)=\left( w_{0}-\sum_{k=1}^{p}c_{k}\left[ \frac{t_{k}}{\beta-1} \int_{0}^{T}\mathscr F\left( s,  w(s),w'(s),\int_{0}^{s}\mathscr G\left( s, \sigma, w(\sigma),w'(\sigma)\right)d\sigma\right) ds \right.\right.\nonumber\\
&\qquad\left.\left.+\int_{0}^{t_{k}}(t_{k}-s)\mathscr F\left( s,  w(s),w'(s),\int_{0}^{s}\mathscr G\left( s, \sigma, w(\sigma),w'(\sigma)\right)d\sigma\right) ds\right] \right)\bigg/\left( 1+\sum_{k=1}^{p}c_{k}\right). 
\end{align*}
\end{small}
Putting this value of $w(0)$ in \eqref{e5}, we get
\begin{small}
\begin{align*}
& w(t)=\left( w_{0}-\sum_{k=1}^{p}c_{k}\left[ \frac{t_{k}}{\beta-1} \int_{0}^{T}\mathscr F\left( s,  w(s),w'(s),\int_{0}^{s}\mathscr G\left( s, \sigma, w(\sigma),w'(\sigma)\right)d\sigma\right) ds \right.\right.\nonumber\\
&\qquad\left.\left.+\int_{0}^{t_{k}}(t_{k}-s)\mathscr F\left( s,  w(s),w'(s),\int_{0}^{s}\mathscr G\left( s, \sigma, w(\sigma),w'(\sigma)\right)d\sigma\right) ds\right] \right)\bigg/\left( 1+\sum_{k=1}^{p}c_{k}\right) \nonumber\\
&\qquad \quad +\frac{t}{\beta-1} \int_{0}^{T}\mathscr F\left( s,  w(s),w'(s),\int_{0}^{s}\mathscr G\left( s, \sigma, w(\sigma),w'(\sigma)\right)d\sigma\right) ds \nonumber\\
&\qquad \qquad+\int_{0}^{t}(t-s)\mathscr F\left( s,  w(s),w'(s),\int_{0}^{s}\mathscr G\left( s, \sigma, w(\sigma),w'(\sigma)\right)d\sigma\right) ds.
\end{align*}
\end{small}
Which is  equation \eqref{e1}. Conversely let $w\in C^{1}(J,\mathbb{R})$ is solution of \eqref{e1} we prove that $w$ satisfied \eqref{1}--\eqref{3}. Differentiating \eqref{e1} with respect to $t$ we get 
\begin{small}
\begin{align}\label{e8}
& w'(t)=\frac{1}{\beta-1} \int_{0}^{T}\mathscr F\left( s,  w(s),w'(s),\int_{0}^{s}\mathscr G\left( s, \sigma, w(\sigma),w'(\sigma)\right)d\sigma\right) ds \nonumber\\
&\qquad \qquad+\int_{0}^{t}\mathscr F\left( s,  w(s),w'(s),\int_{0}^{s}\mathscr G\left( s, \sigma, w(\sigma),w'(\sigma)\right)d\sigma\right) ds.
\end{align}
\end{small}
On differentiating \eqref{e8}, we obtain
\begin{small}
$$w^{''}(t)=\mathscr F\left( t,  w(t),w'(t),\int_{0}^{t}\mathscr G\left( t, s, w(s),w'(s)\right) ds\right),\,t\in J,$$ 
\end{small}
which is \eqref{1}. From \eqref{e8}, we have
\begin{small}
$$w'(0)=\frac{1}{\beta-1} \int_{0}^{T}\mathscr F\left( s,  w(s),w'(s),\int_{0}^{s}\mathscr G\left( s, \sigma, w(\sigma),w'(\sigma)\right)d\sigma\right) ds $$
\end{small}
 and
\begin{small}
\begin{align*}
w'(T)& =\frac{1}{\beta-1} \int_{0}^{T}\mathscr F\left( s,  w(s),w'(s),\int_{0}^{s}\mathscr G\left( s, \sigma, w(\sigma),w'(\sigma)\right)d\sigma\right) ds \nonumber\\
&\qquad \qquad+\int_{0}^{T}\mathscr F\left( s,  w(s),w'(s),\int_{0}^{s}\mathscr G\left( s, \sigma, w(\sigma),w'(\sigma)\right)d\sigma\right) ds.\\
& =\left( \frac{\beta}{\beta-1}\right)  \int_{0}^{T}\mathscr F\left( s,  w(s),w'(s),\int_{0}^{s}\mathscr G\left( s, \sigma, w(\sigma),w'(\sigma)\right)d\sigma\right) ds \\
& = \beta\, w'(0).
\end{align*}
\end{small}
Further, from \eqref{e8}, we have
\begin{small}
\begin{align*}
w(0)& =\left( w_{0}-\sum_{k=1}^{p}c_{k}\left[ \frac{t_{k}}{\beta-1} \int_{0}^{T}\mathscr F\left( s,  w(s),w'(s),\int_{0}^{s}\mathscr G\left( s, \sigma, w(\sigma),w'(\sigma)\right)d\sigma\right) ds \right.\right.\nonumber\\
&\quad\left.\left.+\int_{0}^{t_{k}}(t_{k}-s)\mathscr F\left( s,  w(s),w'(s),\int_{0}^{s}\mathscr G\left( s, \sigma, w(\sigma),w'(\sigma)\right)d\sigma\right) ds\right] \right)\bigg/\left( 1+\sum_{k=1}^{p}c_{k}\right)
\end{align*}
\end{small} 
and for each $i\, (i= 1,2,\dots p)$,
\begin{small}
\begin{align*}
& w(t_{i})=\left( w_{0}-\sum_{k=1}^{p}c_{k}\left[ \frac{t_{k}}{\beta-1} \int_{0}^{T}\mathscr F\left( s,  w(s),w'(s),\int_{0}^{s}\mathscr G\left( s, \sigma, w(\sigma),w'(\sigma)\right)d\sigma\right) ds \right.\right.\nonumber\\
&\qquad\left.\left.+\int_{0}^{t_{k}}(t_{k}-s)\mathscr F\left( s,  w(s),w'(s),\int_{0}^{s}\mathscr G\left( s, \sigma, w(\sigma),w'(\sigma)\right)d\sigma\right) ds\right] \right)\bigg/\left( 1+\sum_{k=1}^{p}c_{k}\right) \nonumber\\
&\qquad \quad +\frac{t_{i}}{\beta-1} \int_{0}^{T}\mathscr F\left( s,  w(s),w'(s),\int_{0}^{s}\mathscr G\left( s, \sigma, w(\sigma),w'(\sigma)\right)d\sigma\right) ds \nonumber\\
&\qquad \qquad+\int_{0}^{t_{i}}(t_{i}-s)\mathscr F\left( s,  w(s),w'(s),\int_{0}^{s}\mathscr G\left( s, \sigma, w(\sigma),w'(\sigma)\right)d\sigma\right) ds .
\end{align*}
\end{small} 
Therefore 
\begin{small}
\begin{align*}
& w(0)+\sum_{i=1}^{p} c_{i} w(t_{i})\\
& = \left( w_{0}-\sum_{k=1}^{p}c_{k}\left[ \frac{t_{k}}{\beta-1} \int_{0}^{T}\mathscr F\left( s,  w(s),w'(s),\int_{0}^{s}\mathscr G\left( s, \sigma, w(\sigma),w'(\sigma)\right)d\sigma\right) ds \right.\right.\nonumber\\
&\qquad\left.\left.+\int_{0}^{t_{k}}(t_{k}-s)\mathscr F\left( s,  w(s),w'(s),\int_{0}^{s}\mathscr G\left( s, \sigma, w(\sigma),w'(\sigma)\right)d\sigma\right) ds\right] \right)\bigg/\left( 1+\sum_{k=1}^{p}c_{k}\right)\\
& \qquad+ \sum_{i=1}^{p} c_{i}\left\lbrace\left( w_{0}-\sum_{k=1}^{p}c_{k}\left[ \frac{t_{k}}{\beta-1} \int_{0}^{T}\mathscr F\left( s,  w(s),w'(s),\int_{0}^{s}\mathscr G\left( s, \sigma, w(\sigma),w'(\sigma)\right)d\sigma\right) ds \right.\right.\right.\nonumber\\
&\qquad\left.\left.\left.+\int_{0}^{t_{k}}(t_{k}-s)\mathscr F\left( s,  w(s),w'(s),\int_{0}^{s}\mathscr G\left( s, \sigma, w(\sigma),w'(\sigma)\right)d\sigma\right) ds\right] \right)\bigg/\left( 1+\sum_{k=1}^{p}c_{k}\right) \right. \nonumber\\
&\left.\qquad +\frac{t_{i}}{\beta-1} \int_{0}^{T}\mathscr F\left( s,  w(s),w'(s),\int_{0}^{s}\mathscr G\left( s, \sigma, w(\sigma),w'(\sigma)\right)d\sigma\right) ds \right.\nonumber\\
&\left. \qquad+\int_{0}^{t_{i}}(t_{i}-s)\mathscr F\left( s,  w(s),w'(s),\int_{0}^{s}\mathscr G\left( s, \sigma, w(\sigma),w'(\sigma)\right)d\sigma\right) ds \right\rbrace\\ 
&= w_{0}.
\end{align*}
\end{small} 
Which is conditions \eqref{3}. This complete the proof.
\end{proof}
\begin{theorem} \label{thm1}
Assume that:
\begin{itemize}
\item[(H1)] Let $\mathscr F \in \left(J \times \mathbb{R}\times \mathbb{R}\times \mathbb{R},\mathbb{R}  \right),\,\mathscr G \in \left(J \times J \times \mathbb{R}\times \mathbb{R},\mathbb{R}  \right) $  and there exist constants $L_{\mathscr F}, L_{\mathscr G}> 0$ such that
\begin{small}
 $$ \left|\mathscr F(t,w_{1},w_{2},w_{3})-\mathscr F(t,v_{1},v_{2},v_{3}) \right|\leq L_{\mathscr F} \left(\sum_{j=1}^{3} \left| w_{j}-v_{j}\right| \right) $$
\end{small}
and
\begin{small}
$$ \left|\mathscr G(t,s,w_{1},w_{2})-\mathscr G(t,s,v_{1},v_{2}) \right|\leq L_{\mathscr G} \left(\sum_{j=1}^{2} \left| w_{j}-v_{j}\right| \right) $$
\end{small}
for all $t,s \in J$ and $w_{j},v_{j}\in \mathbb{R}\, (j=1,2,3)$.
\item[(H2)] There exist a constant $\gamma >0$ such that 
\begin{small}
$$ q= \frac{L_{\mathscr F}}{\gamma}\left(1+\frac{L_{\mathscr G}}{\gamma} \right)\left( 1+\left[ 1+\left\lbrace  T\beta+T\beta\left| \dfrac{\sum\limits_{k=1}^{p}c_{k}}{1+\sum\limits_{k=1}^{p}c_{k}} \right| \right\rbrace  \right]\frac{e^{\gamma\,T}}{\beta-1} \right) <1 .$$
\end{small}
\end{itemize}
Then the VIDNBC \eqref{1}--\eqref{3} has a unique solution in $C^{2}(J, \mathbb{R})$.
\end{theorem}
\begin{proof}
Consider the space $C^{1}(J,\mathbb{R})$ with the norm 
$$ \left\| w\right\|_{1}=\underset{t\in J}{\max}\left\lbrace \frac{\left|w(t) \right|+\left|w'(t) \right| }{e^{\gamma\, t}} \right\rbrace,\,\gamma>0,\,w\in C^{1}(J,\mathbb{R}).  $$
Then $ \left(C^{1}(J,\mathbb{R}),\left\| \cdot\right\|_{1}  \right) $ is a Banach space. Define the operator $\mathcal{P}: (C^{1}(J,\mathbb{R}), \left\| \cdot\right\|_{1}  ) \to (C^{1}(J,\mathbb{R}), \left\| \cdot\right\|_{1}  )$
\begin{small}
\begin{align*}
&\mathcal{P}(w)(t)=\left( w_{0}-\sum_{k=1}^{p}c_{k}\left[ \frac{t_{k}}{\beta-1} \int_{0}^{T}\mathscr F\left( s,  w(s),w'(s),\int_{0}^{s}\mathscr G\left( s, \sigma, w(\sigma),w'(\sigma)\right)d\sigma\right) ds \right.\right.\nonumber\\
&\qquad\qquad\left.\left.+\int_{0}^{t_{k}}(t_{k}-s)\mathscr F\left( s,  w(s),w'(s),\int_{0}^{s}\mathscr G\left( s, \sigma, w(\sigma),w'(\sigma)\right)d\sigma\right) ds\right] \right)\bigg/\left( 1+\sum_{k=1}^{p}c_{k}\right) \nonumber\\
&\qquad \qquad\quad +\frac{t}{\beta-1} \int_{0}^{T}\mathscr F\left( s,  w(s),w'(s),\int_{0}^{s}\mathscr G\left( s, \sigma, w(\sigma),w'(\sigma)\right)d\sigma\right) ds \nonumber\\
&\qquad \qquad\qquad+\int_{0}^{t}(t-s)\mathscr F\left( s,  w(s),w'(s),\int_{0}^{s}\mathscr G\left( s, \sigma, w(\sigma),w'(\sigma)\right)d\sigma\right) ds.
\end{align*}
\end{small}
Then the fixed point of  $ w=\mathcal{P} w,\,\, w\in \left(C^{1}(J,\mathbb{R}),\left\| \cdot\right\|_{1}  \right) $ is the solution of \eqref{1}--\eqref{3}. Let any $w,v \in C^{1}(J,\mathbb{R})$ and $t\in J$. Then
\begin{small}
\begin{align}\label{e9}
&\left|(\mathcal{P}w)(t)-(\mathcal{P}v)(t) \right|\nonumber\\
& \leq\left| \dfrac{\sum\limits_{k=1}^{p}c_{k}}{1+\sum\limits_{k=1}^{p}c_{k}} \right| \left| \frac{t_{k}}{\beta-1} \int_{0}^{T} \left\lbrace  \mathscr F\left( s,  w(s),w'(s),\int_{0}^{s}\mathscr G\left( s, \sigma, w(\sigma),w'(\sigma)\right)d\sigma\right)\right.\right.\nonumber\\
&\quad\left.\left.-F\left( s,  v(s),v'(s),\int_{0}^{s}\mathscr G\left( s, \sigma, v(\sigma),v'(\sigma)\right)d\sigma\right) \right\rbrace  ds \right.\nonumber\\
& \quad\left.+\int_{0}^{t_{k}}(t_{k}-s)\left\lbrace  \mathscr F\left( s,  w(s),w'(s),\int_{0}^{s}\mathscr G\left( s, \sigma, w(\sigma),w'(\sigma)\right)d\sigma\right)\right.\right.\nonumber\\
& \quad\left.\left.-\mathscr F\left( s,  v(s),v'(s),\int_{0}^{s}\mathscr G\left( s, \sigma, v(\sigma),v'(\sigma)\right)d\sigma\right)\right\rbrace   ds\right| \nonumber\\
&+ \left| \frac{t}{\beta-1} \int_{0}^{T}\left\lbrace \mathscr F\left( s,  w(s),w'(s),\int_{0}^{s}\mathscr G\left( s, \sigma, w(\sigma),w'(\sigma)\right)d\sigma\right)\right.\right.\nonumber\\
&\quad\left.\left.-F\left( s,  v(s),v'(s),\int_{0}^{s}\mathscr G\left( s, \sigma, v(\sigma),v'(\sigma)\right)d\sigma\right) \right\rbrace  ds \right.\nonumber\\
& \quad\left.+\int_{0}^{t}(t-s)\left\lbrace  \mathscr F\left( s,  w(s),w'(s),\int_{0}^{s}\mathscr G\left( s, \sigma, w(\sigma),w'(\sigma)\right)d\sigma\right)\right.\right.\nonumber\\
& \quad\left.\left.-\mathscr F\left( s,  v(s),v'(s),\int_{0}^{s}\mathscr G\left( s, \sigma, v(\sigma),v'(\sigma)\right)d\sigma\right)\right\rbrace   ds\right| \nonumber\\
& \leq\left| \dfrac{\sum\limits_{k=1}^{p}c_{k}}{1+\sum\limits_{k=1}^{p}c_{k}} \right| \left| \frac{T}{\beta-1} \int_{0}^{T}\left\lbrace \mathscr F\left( s,  w(s),w'(s),\int_{0}^{s}\mathscr G\left( s, \sigma, w(\sigma),w'(\sigma)\right)d\sigma\right)\right.\right.\nonumber\\
&\quad\left.\left.-F\left( s,  v(s),v'(s),\int_{0}^{s}\mathscr G\left( s, \sigma, v(\sigma),v'(\sigma)\right)d\sigma\right) \right\rbrace  ds \right.\nonumber\\
& \quad\left.+\int_{0}^{T}T\left\lbrace  \mathscr F\left( s,  w(s),w'(s),\int_{0}^{s}\mathscr G\left( s, \sigma, w(\sigma),w'(\sigma)\right)d\sigma\right)\right.\right.\nonumber\\
& \quad\left.\left.-\mathscr F\left( s,  v(s),v'(s),\int_{0}^{s}\mathscr G\left( s, \sigma, v(\sigma),v'(\sigma)\right)d\sigma\right)\right\rbrace   ds\right| \nonumber\\
&+ \left| \frac{T}{\beta-1} \int_{0}^{T}\left\lbrace  \mathscr F\left( s,  w(s),w'(s),\int_{0}^{s}\mathscr G\left( s, \sigma, w(\sigma),w'(\sigma)\right)d\sigma\right)\right.\right.\nonumber\\
&\quad\left.\left.-F\left( s,  v(s),v'(s),\int_{0}^{s}\mathscr G\left( s, \sigma, v(\sigma),v'(\sigma)\right)d\sigma\right) \right\rbrace  ds \right.\nonumber\\
& \quad\left.+\int_{0}^{T}T\left\lbrace  \mathscr F\left( s,  w(s),w'(s),\int_{0}^{s}\mathscr G\left( s, \sigma, w(\sigma),w'(\sigma)\right)d\sigma\right)\right.\right.\nonumber\\
& \quad\left.\left.-\mathscr F\left( s,  v(s),v'(s),\int_{0}^{s}\mathscr G\left( s, \sigma, v(\sigma),v'(\sigma)\right)d\sigma\right)\right\rbrace  ds\right| \nonumber\\
&= \left( 1+\left| \dfrac{\sum\limits_{k=1}^{p}c_{k}}{1+\sum\limits_{k=1}^{p}c_{k}} \right|\right)\left| \frac{T}{\beta-1} \int_{0}^{T}\left\lbrace  \mathscr F\left( s,  w(s),w'(s),\int_{0}^{s}\mathscr G\left( s, \sigma, w(\sigma),w'(\sigma)\right)d\sigma\right)\right.\right.\nonumber\\
&\quad\left.\left.-\mathscr F\left( s,  v(s),v'(s),\int_{0}^{s}\mathscr G\left( s, \sigma, v(\sigma),v'(\sigma)\right)d\sigma\right) \right\rbrace  ds \right.\nonumber\\
& \quad\left.+\int_{0}^{T}T\left\lbrace \mathscr F\left( s,  w(s),w'(s),\int_{0}^{s}\mathscr G\left( s, \sigma, w(\sigma),w'(\sigma)\right)d\sigma\right)\right.\right.\nonumber\\
& \quad\left.\left.-\mathscr F\left( s,  v(s),v'(s),\int_{0}^{s}\mathscr G\left( s, \sigma, v(\sigma),v'(\sigma)\right)d\sigma\right)\right\rbrace   ds\right| \nonumber\\
& \leq  \left( 1+\left| \dfrac{\sum\limits_{k=1}^{p}c_{k}}{1+\sum\limits_{k=1}^{p}c_{k}} \right|\right) \left[  \frac{T}{\beta-1} \int_{0}^{T}L_{\mathscr F}\,e^{\gamma s}e^{-\gamma s}\left\lbrace  \left|w(s)-v(s)\right| +\left|w'(s)-v'(s)\right|\right\rbrace ds\right.\nonumber\\
& \quad\left.+ \frac{T}{\beta-1} \int_{0}^{T} \int_{0}^{s} L_{\mathscr F}\,L_{\mathscr G}\,e^{\gamma \sigma}e^{-\gamma \sigma}\left\lbrace  \left|w(\sigma)-v(\sigma)\right| +\left|w'(\sigma)-v'(\sigma)\right|\right\rbrace d\sigma\,ds \right.\nonumber\\
& \quad\left. +T\int_{0}^{T} L_{\mathscr F}\,e^{\gamma s}e^{-\gamma s}\left\lbrace  \left|w(s)-v(s)\right| +\left|w'(s)-v'(s)\right|\right\rbrace ds\right.\nonumber\\
& \quad\left.  + T\int_{0}^{T} \int_{0}^{s}L_{\mathscr F}\,L_{\mathscr G}\,e^{\gamma \sigma}e^{-\gamma \sigma}\left\lbrace  \left|w(\sigma)-v(\sigma)\right| +\left|w'(\sigma)-v'(\sigma)\right|\right\rbrace d\sigma\,ds
\right]  \nonumber\\
& \leq  \left( 1+\left| \dfrac{\sum\limits_{k=1}^{p}c_{k}}{1+\sum\limits_{k=1}^{p}c_{k}} \right|\right) \left\lbrace  \frac{T}{\beta-1}L_{\mathscr F} \left(\frac{e^{\gamma T}}{\gamma}- \frac{1}{\gamma}\right) \left\|w-v \right\| _{1}+ \frac{T}{\beta-1} L_{\mathscr F}\,L_{\mathscr G}  \left(\frac{e^{\gamma T}}{\gamma^{2}}- \frac{1}{\gamma^{2}}-\frac{T}{\gamma}\right) \left\|w-v \right\| _{1} \right.\nonumber\\
& \quad\quad\left. +T\,L_{\mathscr F}  \left(\frac{e^{\gamma T}}{\gamma}- \frac{1}{\gamma}\right) \left\|w-v \right\| _{1}+T\,L_{\mathscr F}\,L_{\mathscr G}  \left(\frac{e^{\gamma T}}{\gamma^{2}}- \frac{1}{\gamma^{2}}-\frac{T}{\gamma}\right) \left\|w-v \right\| _{1}
\right\rbrace  \nonumber\\
& \leq  \left( 1+\left| \dfrac{\sum\limits_{k=1}^{p}c_{k}}{1+\sum\limits_{k=1}^{p}c_{k}} \right|\right) \left\lbrace   \frac{T}{\beta-1}L_{\mathscr F} \frac{e^{\gamma T}}{\gamma}  + \frac{T}{\beta-1} L_{\mathscr F}\,L_{\mathscr G}  \frac{e^{\gamma T}}{\gamma^{2}}+T\,L_{\mathscr F}  \frac{e^{\gamma T}}{\gamma}+ T\,L_{\mathscr F}\,L_{\mathscr G}  \frac{e^{\gamma T}}{\gamma^{2}} 
\right\rbrace \left\|w-v \right\| _{1} \nonumber\\
&=\frac{L_{\mathscr F}}{\gamma}\left(1+\frac{L_{\mathscr G}}{\gamma} \right) \left\lbrace  T\beta+T\beta\left| \dfrac{\sum\limits_{k=1}^{p}c_{k}}{1+\sum\limits_{k=1}^{p}c_{k}} \right| \right\rbrace \frac{e^{\gamma\,T}}{\beta-1}\left\|w-v \right\| _{1}.
\end{align}
\end{small}
Similarly we have
\begin{small}
\begin{align}\label{e10}
\left|(\mathcal{P}w)^{'}(t)-(\mathcal{P}v)^{'}(t) \right|& \leq\left| \frac{1}{\beta-1} \int_{0}^{T}\left[ \mathscr F\left( s,  w(s),w'(s),\int_{0}^{s}\mathscr G\left( s, \sigma, w(\sigma),w'(\sigma)\right)d\sigma\right)\right.\right.\nonumber\\
&\quad\left.\left.-F\left( s,  v(s),v'(s),\int_{0}^{s}\mathscr G\left( s, \sigma, v(\sigma),v'(\sigma)\right)d\sigma\right) \right] ds\right| \nonumber\\
& +  \left| \int_{0}^{t}\left[ \mathscr F\left( s,  w(s),w'(s),\int_{0}^{s}\mathscr G\left( s, \sigma, w(\sigma),w'(\sigma)\right)d\sigma\right)\right.\right.\nonumber\\
&\quad\left.\left.-\mathscr F\left( s,  v(s),v'(s),\int_{0}^{s}\mathscr G\left( s, \sigma, v(\sigma),v'(\sigma)\right)d\sigma\right) \right] ds\right|\nonumber\\
& \leq \frac{L_{\mathscr F}}{\gamma}\left(1+\frac{L_{\mathscr G}}{\gamma} \right)\left(\frac{e^{\gamma\,T}}{\beta-1}+e^{\gamma\, t} \right)  \left\|w-v \right\| _{1}.
\end{align}
\end{small}
Thus from \eqref{e9} and \eqref{e10} we have
\begin{small}
\begin{align}\label{e11}
&\left|(\mathcal{P}w)(t)-(\mathcal{P}v)(t) \right|+\left|(\mathcal{P}w)^{'}(t)-(\mathcal{P}v)^{'}(t) \right|\nonumber\\
& \leq \frac{L_{\mathscr F}}{\gamma}\left(1+\frac{L_{\mathscr G}}{\gamma} \right)\left\lbrace \left(   T\beta+T\beta\left| \dfrac{\sum\limits_{k=1}^{p}c_{k}}{1+\sum\limits_{k=1}^{p}c_{k}} \right|  \right) \frac{e^{\gamma\,T}}{\beta-1}+ \frac{e^{\gamma\,T}}{\beta-1}+e^{\gamma\, t}\right\rbrace\left\|w-v \right\| _{1},\, t \in J.
\end{align}
\end{small}
Therefore
\begin{small}
\begin{align*}
\left\|\mathcal{P}w-\mathcal{P}v \right\|_{1}
&=\underset{t\in J}{\max}\,\,\frac{1}{e^{\gamma\, t}} \left\lbrace \left|(\mathcal{P}w)(t)-(\mathcal{S}w)(t) \right|+\left|(\mathcal{P}w)^{'}(t)-(\mathcal{S}w)(t)\right| \right\rbrace \nonumber\\
&\leq\frac{L_{\mathscr F}}{\gamma}\left(1+\frac{L_{\mathscr G}}{\gamma} \right)\left( 1+\left[ 1+\left\lbrace  T\beta+T\beta\left| \dfrac{\sum\limits_{k=1}^{p}c_{k}}{1+\sum\limits_{k=1}^{p}c_{k}} \right| \right\rbrace  \right]\frac{e^{\gamma\,T}}{\beta-1} \right)\left\|w-v \right\| _{1}.
\end{align*}
\end{small}
Choose $ \gamma>0$ such that 
\begin{small}
$$ \frac{L_{\mathscr F}}{\gamma}\left(1+\frac{L_{\mathscr G}}{\gamma} \right)\left( 1+\left[ 1+\left\lbrace  T\beta+T\beta\left| \dfrac{\sum\limits_{k=1}^{p}c_{k}}{1+\sum\limits_{k=1}^{p}c_{k}} \right| \right\rbrace  \right]\frac{e^{\gamma\,T}}{\beta-1} \right)<1.$$
\end{small} 
Hence $\mathcal{P}$ is contraction  and by Banach contraction principle $\mathcal{P}$ has a fixed point $\bar{w}$ in  $ \left(C^{1}(J,\mathbb{R}),\left\| \cdot\right\|_{1}  \right) $ which is solution of VIDNBC \eqref{1}--\eqref{3}.
\end{proof}
\section{Dependency of solutions via PO:}
In this section we give the dependence of solution of problem \eqref{1}--\eqref{3} on the functions involved in the right hand side of equations \eqref{1}--\eqref{2} and on the initial nonlocal data via Picard operator theory. Note that \cite{Wang} every contraction operator is a Picard operator.
\begin{small}
\begin{align}
& w^{''}(t)=\tilde{\mathscr F}\left( t,  w(t),w'(t),\int_{0}^{t}\tilde{\mathscr G}\left( t, s, w(s),w'(s)\right) ds\right),\,t\in J=[0,T],\, T>0,\label{4}\\
& w(0)+\sum_{k=1}^{p} c_{k} w(t_{k})=\tilde{w}_{0},\label{5}\\
& w'(T)= \beta\, w'(0),\,\, 1<\beta<+\infty. \label{6}
\end{align}
\end{small}
Then its equivalent Volterra integral equation is 
\begin{small}
\begin{align}\label{e13}
& w(t)=\left(\tilde{w}_{0}-\sum_{k=1}^{p}c_{k}\left[ \frac{t_{k}}{\beta-1} \int_{0}^{T}\tilde{\mathscr F}\left( s,  w(s),w'(s),\int_{0}^{s}\tilde{\mathscr G}\left( s, \sigma, w(\sigma),w'(\sigma)\right)d\sigma\right) ds \right.\right.\nonumber\\
&\qquad\left.\left.+\int_{0}^{t_{k}}(t_{k}-s)\tilde{\mathscr F}\left( s,  w(s),w'(s),\int_{0}^{s}\tilde{\mathscr G}\left( s, \sigma, w(\sigma),w'(\sigma)\right)d\sigma\right) ds\right] \right)\bigg/\left( 1+\sum_{k=1}^{p}c_{k}\right) \nonumber\\
&\qquad \quad +\frac{t}{\beta-1} \int_{0}^{T}\tilde{\mathscr F}\left( s,  w(s),w'(s),\int_{0}^{s}\tilde{\mathscr G}\left( s, \sigma, w(\sigma),w'(\sigma)\right)d\sigma\right) ds \nonumber\\
&\qquad \qquad+\int_{0}^{t}(t-s)\tilde{\mathscr F}\left( s,  w(s),w'(s),\int_{0}^{s}\tilde{\mathscr G}\left( s, \sigma, w(\sigma),w'(\sigma)\right)d\sigma\right) ds.
\end{align}
\end{small}
\begin{theorem}
Suppose the following:
\begin{itemize}
\item[$(H1)^{'}$]All the conditions in Theorem \ref{thm1} are satisfied and $ w^{\ast}\in C^{1}\left( J, \mathbb{R}\right) $ is the unique solution of the integral equation \eqref{e1}.
\item[$(H2)^{'}$] There exists  $L_{\tilde{\mathscr F}}, L_{\tilde{\mathscr G}}> 0$  such that
\begin{small}
 $$ \left|\tilde{\mathscr F} (t,w_{1},w_{2},w_{3})-\tilde{\mathscr F}(t,v_{1},v_{2},v_{3}) \right|\leq L_{\tilde{\mathscr F}} \left(\sum_{j=1}^{3} \left| w_{j}-v_{j}\right| \right) $$
\end{small}
and
\begin{small}
$$ \left|\tilde{\mathscr G}(t,s,w_{1},w_{2})-\tilde{\mathscr G}(t,s,v_{1},v_{2}) \right|\leq L_{\tilde{\mathscr G}} \left(\sum_{j=1}^{2} \left| w_{j}-v_{j}\right| \right) $$
\end{small}
for all $t,s \in J$ and $w_{j},v_{j}\in \mathbb{R}\, (j=1,2,3)$.
\item[$(H3)^{'}$] There exists a function $\mu(\cdot)\in L^{1}\left(J,\mathbb{R}_{+} \right) \cap C\left(J,\mathbb{R}_{+} \right) $ such that
$$ \left|\mathscr F(t,u,v,w)-\tilde{\mathscr F}(t,u,v,\tilde{w}) \right|\leq \mu(t)  $$
for all $t \in J$ and $u,\,v,\,w,\,\tilde{w}\in \mathbb{R}$.
\end{itemize}
Then if $v^{\ast}$ is the solution of integral equations \eqref{e13} then 
\begin{small}
$$ \left\|w^{\ast}-v^{\ast} \right\|_{1}\leq \frac{\left| \frac{1}{1+\sum\limits_{k=1}^{p}c_{k}}\right| \left|w_{0}-\tilde{w}_{0} \right|+ \frac{\beta  L_{\mu}}{\beta-1}\left[ 1+\left( 1 +\left| \frac{\sum\limits_{k=1}^{p}c_{k}}{1+\sum\limits_{k=1}^{p}c_{k}}\right|\right) T\right]  }{1-q} ,$$
\end{small}
where $L_{\mu}=\int_{0}^{T} \mu(s)\, ds$.
\end{theorem}
\begin{proof}
Consider the  operators
$\mathcal{P},\mathcal{S}: (C^{1}(J,\mathbb{R}), \left\| \cdot\right\|_{1}  ) \to (C^{1}(J,\mathbb{R}), \left\| \cdot\right\|_{1}  ) $ defined by
\begin{small}
\begin{align*}
\mathcal{P}(w)(t)&=\left( w_{0}-\sum_{k=1}^{p}c_{k}\left[ \frac{t_{k}}{\beta-1} \int_{0}^{T}\mathscr F\left( s,  w(s),w'(s),\int_{0}^{s}\mathscr G\left( s, \sigma, w(\sigma),w'(\sigma)\right)d\sigma\right) ds \right.\right.\nonumber\\
&\qquad\left.\left.+\int_{0}^{t_{k}}(t_{k}-s)\mathscr F\left( s,  w(s),w'(s),\int_{0}^{s}\mathscr G\left( s, \sigma, w(\sigma),w'(\sigma)\right)d\sigma\right) ds\right] \right)\bigg/\left( 1+\sum_{k=1}^{p}c_{k}\right) \nonumber\\
&\qquad \quad +\frac{t}{\beta-1} \int_{0}^{T}\mathscr F\left( s,  w(s),w'(s),\int_{0}^{s}\mathscr G\left( s, \sigma, w(\sigma),w'(\sigma)\right)d\sigma\right) ds \nonumber\\
&\qquad \qquad+\int_{0}^{t}(t-s)\mathscr F\left( s,  w(s),w'(s),\int_{0}^{s}\mathscr G\left( s, \sigma, w(\sigma),w'(\sigma)\right)d\sigma\right) ds.
\end{align*}
\end{small}
and
\begin{small}
\begin{align*}
\mathcal{S}(w)(t)&=\left(\tilde{w}_{0}-\sum_{k=1}^{p}c_{k}\left[ \frac{t_{k}}{\beta-1} \int_{0}^{T}\tilde{\mathscr F}\left( s,  w(s),w'(s),\int_{0}^{s}\tilde{\mathscr G}\left( s, \sigma, w(\sigma),w'(\sigma)\right)d\sigma\right) ds \right.\right.\nonumber\\
&\qquad\left.\left.+\int_{0}^{t_{k}}(t_{k}-s)\tilde{\mathscr F}\left( s,  w(s),w'(s),\int_{0}^{s}\tilde{\mathscr G}\left( s, \sigma, w(\sigma),w'(\sigma)\right)d\sigma\right) ds\right] \right)\bigg/\left( 1+\sum_{k=1}^{p}c_{k}\right) \nonumber\\
&\qquad \quad +\frac{t}{\beta-1} \int_{0}^{T}\tilde{\mathscr F}\left( s,  w(s),w'(s),\int_{0}^{s}\tilde{\mathscr G}\left( s, \sigma, w(\sigma),w'(\sigma)\right)d\sigma\right) ds \nonumber\\
&\qquad \qquad+\int_{0}^{t}(t-s)\tilde{\mathscr F}\left( s,  w(s),w'(s),\int_{0}^{s}\tilde{\mathscr G}\left( s, \sigma, w(\sigma),w'(\sigma)\right)d\sigma\right) ds.
\end{align*}
\end{small}
By condition $(H1)^{'}\,\,\, \mathcal{P}$ is a contraction with contraction constant $q$. Let $\mathbf{F}_{\mathcal{P}}=\{w^{\ast}\}$. Following same steps in the proof of the Theorem \ref{thm1}, the operator that $\mathcal{S}$ is contraction with contraction constant
\begin{small}
$$ \tilde{q}=\frac{L_{\tilde{\mathscr F}}}{\gamma}\left(1+\frac{L_{\tilde{\mathscr G}}}{\gamma} \right)\left( 1+\left[ 1+\left\lbrace  T\beta+T\beta\left| \dfrac{\sum\limits_{k=1}^{p}c_{k}}{1+\sum\limits_{k=1}^{p}c_{k}} \right| \right\rbrace  \right]\frac{e^{\gamma\,T}}{\beta-1} \right)<1.$$
\end{small} 
Hence the VIDNBC \eqref{4}--\eqref{6} has a unique solution. Let $\mathbf{F}_{\mathcal{S}}= \{v^{\ast}\}$. For any $w \in  \left(C^{1}(J,\mathbb{R}),\left\| \cdot\right\|_{1}  \right)$. Then any $t\in J$, we have
\begin{small}
\begin{align}\label{e14}
&\left|(\mathcal{P}w)(t)-(\mathcal{S}w)(t) \right|\nonumber\\
&\leq\left| \dfrac{1}{1+\sum\limits_{k=1}^{p}c_{k}}\right|\left|w_{0}-\tilde{w}_{0} \right|+\left( 1+\left| \dfrac{\sum\limits_{k=1}^{p}c_{k}}{1+\sum\limits_{k=1}^{p}c_{k}} \right|\right)\left| \frac{T}{\beta-1} \int_{0}^{T}\left\lbrace \mathscr F\left( s,  w(s),w'(s),\int_{0}^{s}\mathscr G\left( s, \sigma, w(\sigma),w'(\sigma)\right)d\sigma\right)\right.\right.\nonumber\\
&\quad\left.\left.-\tilde{\mathscr F}\left( s,  w(s),w'(s),\int_{0}^{s}\tilde{\mathscr G}\left( s, \sigma, w(\sigma),w'(\sigma)\right)d\sigma\right) \right\rbrace  ds \right.\nonumber\\
& \quad\left.+\int_{0}^{T}T\left\lbrace  \mathscr F\left( s,  w(s),w'(s),\int_{0}^{s}\mathscr G\left( s, \sigma, w(\sigma),w'(\sigma)\right)d\sigma\right)\right.\right.\nonumber\\
& \quad\left.\left.-\tilde{\mathscr F}\left( s,  w(s),w'(s),\int_{0}^{s}\tilde{\mathscr G}\left( s, \sigma, w(\sigma),w'(\sigma)\right)d\sigma\right)\right\rbrace   ds\right| \nonumber\\
&\leq \left| \dfrac{1}{1+\sum\limits_{k=1}^{p}c_{k}}\right| \left|w_{0}-\tilde{w}_{0} \right| +\left( 1+\left| \dfrac{\sum\limits_{k=1}^{p}c_{k}}{1+\sum\limits_{k=1}^{p}c_{k}} \right|\right)\left\lbrace \frac{T}{\beta-1} \int_{0}^{T} \mu(s)ds+ T \int_{0}^{T} \mu(s)ds\right\rbrace\\ 
&\leq \left| \dfrac{1}{1+\sum\limits_{k=1}^{p}c_{k}}\right| \left|w_{0}-\tilde{w}_{0} \right| +\left( 1+\left| \dfrac{\sum\limits_{k=1}^{p}c_{k}}{1+\sum\limits_{k=1}^{p}c_{k}} \right|\right)\frac{\beta TL_{\mu}}{\beta-1} .
\end{align}
\end{small}
Similarly we have
\begin{small}
\begin{align}\label{e15}
\left|(\mathcal{P}w)^{'}(t)-(\mathcal{S}w)^{'}(t) \right|& \leq\left| \frac{1}{\beta-1} \int_{0}^{T}\left[ \mathscr F\left( s,  w(s),w'(s),\int_{0}^{s}\mathscr G\left( s, \sigma, w(\sigma),w'(\sigma)\right)d\sigma\right)\right.\right.\nonumber\\
&\quad\left.\left.- \tilde{\mathscr F}\left( s,  w(s),w'(s),\int_{0}^{s}\tilde{\mathscr G}\left( s, \sigma, w(\sigma),w'(\sigma)\right)d\sigma\right) \right] ds\right| \nonumber\\
& +  \left| \int_{0}^{t}\left[ \mathscr F\left( s,  w(s),w'(s),\int_{0}^{s}\mathscr G\left( s, \sigma, w(\sigma),w'(\sigma)\right)d\sigma\right)\right.\right.\nonumber\\
&\quad\left.\left.- \tilde{\mathscr F}\left( s,  w(s),w'(s),\int_{0}^{s}\tilde{\mathscr G}\left( s, \sigma, w(\sigma),w'(\sigma)\right)d\sigma\right) \right] ds\right|\nonumber\\
&\leq\frac{1}{\beta-1} \int_{0}^{T} \mu(s)ds+  \int_{0}^{T} \mu(s)ds\\
& \leq \frac{\beta  L_{\mu}}{\beta-1}.
\end{align}
\end{small}
From \eqref{e14} and \eqref{e14}
\begin{small}
\begin{align*}
\left\| \mathcal{P}w- \mathcal{S}w\right\|_{1}&=\underset{t\in J}{\max}\,\,\frac{1}{e^{\gamma\, t}} \left\lbrace \left|(\mathcal{P}w)(t)-(\mathcal{S}w)(t) \right|+\left|(\mathcal{P}w)^{'}(t)-(\mathcal{S}w)^{'}(t) \right| \right\rbrace \\
& \leq \left| \frac{1}{1+\sum\limits_{k=1}^{p}c_{k}}\right| \left|w_{0}-\tilde{w}_{0} \right|+ \frac{\beta  L_{\mu}}{\beta-1}\left[ 1+\left( 1 +\left| \frac{\sum\limits_{k=1}^{p}c_{k}}{1+\sum\limits_{k=1}^{p}c_{k}}\right|\right) T\right]. 
\end{align*}
\end{small}
Using the Theorem \ref{B1} we get following inequality 
\begin{small}
\begin{align}\label{e16}
\left\|w^{\ast}-v^{\ast} \right\|_{1}\leq \frac{\left| \frac{1}{1+\sum\limits_{k=1}^{p}c_{k}}\right| \left|w_{0}-\tilde{w}_{0} \right|+ \frac{\beta L_{\mu}}{\beta-1}\left[ 1+\left( 1 +\left| \frac{\sum\limits_{k=1}^{p}c_{k}}{1+\sum\limits_{k=1}^{p}c_{k}}\right|\right) T\right]  }{1-q} .
\end{align}
\end{small}
\end{proof}
\begin{rem}
\begin{itemize}
\item[{\rm (i)}] From inequality \eqref{e16} it follows that the solutions of the VIDNBC \eqref{1}--\eqref{3} depends on the  initial nonlocal data and  functions involved on the right hand side of equation.
\item[{\rm(ii)}] If $\mu(t)=0 $ then $L_{\mu}=0$. In this case  inequality \eqref{e16} gives dependency of solutions on initial nonlocal data.
\item[{\rm(iii)}] If $w_{0}=\tilde{w}_{0}$ then inequality \eqref{e16} gives dependency of solutions on functions involved on the right hand side of equation.
\item[{\rm(iv)}]  If $\mu(t)=0 $ and $w_{0}=\tilde{w}_{0}$  then $L_{\mu}=0$, in this case  inequality \eqref{e16}  gives uniqueness of solution.
\end{itemize}
\end{rem}
\section{Examples:}
\begin{ex}
Consider the second order Volterra integrodifferential equations with nonlocal and boundary conditions:
\begin{small}
\begin{align} 
& w^{''}(t) = 0.010540+\frac{1}{10} \sin\left(\frac{1}{10} \right)- \frac{\cos(w(t))}{1000}- \frac{\sin(w'(t))}{100}\nonumber\\
&\qquad \qquad\qquad  +\frac{1}{100}\int_0^t \frac{1}{10} \left[ w(s)-\frac{e^{\frac{s}{10}}}{10} \sin (w(s))+\frac{e^{\frac{s}{10}}}{10} \cos (w'(s))\right] \,ds,\, ~t \in J=[0,1],\label{7} \\
& w(0)+ w(t_{1})+ w(t_{2})- w(t_{2})+
(0) w(t_{4})+ w(t_{5})=3.10,\,\,\label{8}\\
& w'(1)= e^{\frac{1}{10}}\, w'(0).\,\label{9}
\end{align}
\end{small}
\end{ex}
Comparing with the equation \eqref{1}--\eqref{3} we have  $T=1,\,\beta=e^{\frac{1}{10}},\, t_{1}=0.2,\,t_{2}=0.4,\,t_{3}=0.6,\,t_{4}=0.8,\,t_{5}=1 $ and   $c_{1}=c_{2}=1,\,c_{3}=-1,\,c_{4}=0,\,c_{5}=1$ such that $\sum\limits_{k=1}^{5} c_{k} =2 \not = -1$.
\begin{enumerate}
\item[(i)] Define $\mathscr G:[0,1]\times[0,1]\times\mathbb{R}\times\mathbb{R} \to \mathbb{R} $ by
\begin{small}
$$ \mathscr G\left( t, s, w(s),w'(s)\right) = \frac{1}{10} \left[ w(s)-\frac{e^{\frac{s}{10}}}{10} \sin (w(s))+\frac{e^{\frac{s}{10}}}{10} \cos (w'(s))\right].$$
\end{small}
Then for any $ t,s\in [0,1]$ and $w_{1},w_{2},v_{1},v_{2}\in \mathbb{R}$, we have
\begin{small}
\begin{align*}
\left| \mathscr G (t,s ,w_{1},w_{2})- \mathscr G(t,s ,v_{1},v_{2})\right|
&\leq\frac{1}{10}\left|w_{1}- v_{1}\right| +\frac{e^{\frac{s}{10}}}{10}\left|\sin w_{1}- \sin v_{1}\right|+\frac{e^{\frac{s}{10}}}{10} \left| \cos w_{2}-\cos v_{2}\right|\\
&\leq\frac{1}{10}\left|w_{1}- v_{1}\right| +\frac{e^{\frac{1}{10}}}{10}\left|w_{1}- v_{1}\right|+ \frac{e^{\frac{1}{10}}}{10} \left|  w_{2}- v_{2} \right|\\
& \leq \frac{1+e^{\frac{1}{10}}}{10} \left\lbrace \left|w_{1}- v_{1}\right| + \left| w_{2}- v_{2}\right| \right\rbrace. 
\end{align*}
\end{small}
\item[(ii)] Define $\mathscr F:[0,1] \times \mathbb{R} \times \mathbb{R}\times \mathbb{R} \to \mathbb{R} $ by
\begin{small}
\begin{align*}
&\mathscr F\left( t,  w(t),w'(t),\int_{0}^{t}\mathscr G\left( t, s, w(s),w'(s)\right) ds\right)\\
&= 0.010540+\frac{1}{10} \sin\left(\frac{1}{10} \right)- \frac{\cos(w(t))}{1000}- \frac{\sin(w'(t))}{100}\nonumber\\
&\qquad +\frac{1}{100}\int_0^t \frac{1}{10} \left[ w(s)-\frac{e^{\frac{s}{10}}}{10} \sin (w(s))+\frac{e^{\frac{s}{10}}}{10} \cos (w'(s))\right] \,ds.\\
& =0.010540+\frac{1}{10} \sin\left(\frac{1}{10} \right)- \frac{\cos(w(t))}{1000}- \frac{\sin(w'(t))}{100}+\int_0^t  G\left( t, s, w(s),w'(s)\right) ds.
\end{align*}
\end{small}
\end{enumerate}
Then for any $ t\in [0,1]$ and $w_{1},w_{2},w_{3},v_{1},v_{2},v_{3}\in \mathbb{R}$, we have
\begin{small}
\begin{align*}
\left| \mathscr F (t,w_{1},w_{2},w_{3})- \mathscr F(t ,v_{1},v_{2},v_{3})\right|
&\leq\frac{1}{1000}\left|\cos w_{1}- \cos v_{1}\right| +\frac{1}{100}\left|\sin w_{2}- \sin v_{2}\right|+\frac{1}{100}\left|w_{3}- v_{3}\right|\\
& \leq \frac{1}{100} \left\lbrace \left|w_{1}- v_{1}\right| + \left| w_{2}- v_{2}\right| +\left|w_{3}- v_{3}\right|\right\rbrace .
\end{align*}
\end{small}
We have proved that $L_{ \mathscr F}$ and $L_{ \mathscr G}$  satisfied the conditions (H1) with $L_{ \mathscr F}=\frac{1}{100},\,\text{and}\,\,L_{ \mathscr G}=\frac{1+e^{\frac{1}{10}}}{10}$.  
\begin{small}
\begin{align*}
q& = \frac{L_{\mathscr F}}{\gamma}\left(1+\frac{L_{\mathscr G}}{\gamma} \right)\left( 1+\left[ 1+\left\lbrace  T\beta+T\beta\left| \dfrac{\sum\limits_{k=1}^{p}c_{k}}{1+\sum\limits_{k=1}^{p}c_{k}} \right| \right\rbrace  \right]\frac{e^{\gamma\,T}}{\beta-1} \right)\\
&=\frac{\frac{1}{100}}{\gamma}\left(1+\frac{\frac{1+e^{\frac{1}{10}}}{10}}{\gamma} \right)\left( 1+\left[ 1+\left\lbrace  1\,e^{\frac{1}{10}} +1\,e^{\frac{1}{10}} \left| \frac{2}{1+2} \right| \right\rbrace  \right]\frac{e^{\gamma\,1}}{e^{\frac{1}{10}}-1} \right).
\end{align*}
\end{small}
Note for $ \gamma = 1,$ we have
\begin{small}
$$q=\frac{1}{100}\left(1+ \frac{1+e^{\frac{1}{10}}}{10} \right)\left( 1+\left[ 1+\left\lbrace  e^{\frac{1}{10}}+e^{\frac{1}{10}}\left| \frac{2}{3} \right| \right\rbrace  \right]\frac{e}{e^{\frac{1}{10}}-1} \right) = 0.901278<1.$$
\end{small}
Thus all the assumptions of the Theorem \ref{thm1} are satisfied. Applying the Theorem \ref{thm1}, the problem \eqref{7}--\eqref{9} has unique solution on $[0,1]$. One can verify that $w(t)=e^{\frac{t}{10}},\, t\in [0,1]$ is the unique solution  \eqref{7}--\eqref{9}.
\begin{ex}
Consider the following second order  Volterra integrodifferential  equations:
\begin{small}
\begin{align} 
& w^{''}(t) =\frac{2}{10}- \frac{t^2}{1000}-\frac{(9-t)}{1000}+ \frac{\cos(w(t))}{100}- \frac{w'(t)}{100}\nonumber\\
& \qquad\qquad\qquad +\frac{1}{100}\int_0^t  \left[ \left( \frac{1+2s}{10}\right) \sin (w(s))+ w'(s)\right] \,ds,\,t \in J= [0,2],\label{10}
\end{align}
\end{small}
subject to
\begin{small}
\begin{align}
& w(0)+ w(t_{1})+w(t_{2})+w(t_{3})+w(t_{4})=1.35,\,\,\label{11}\\
& w'(2)= 5\, w'(0),\,\label{12}
\end{align}
\end{small}
\end{ex}
Comparing with the equation \eqref{1}--\eqref{3} we have  $T=2,\,\beta=5,\, t_{1}=0.5,\,t_{2}=1,\,t_{3}=1.5,\,t_{4}=2$ and   $c_{1}=c_{2}=c_{3}=c_{4}=1$ such that $\sum\limits_{k=1}^{4} c_{k} =4 \not = -1$.\\
Define $\mathscr G:[0,2]\times[0,2]\times\mathbb{R}\times\mathbb{R} \to \mathbb{R} $ by
\begin{small}
$$ \mathscr G\left( t, s, w(s),w'(s)\right) = \left[ \left( \frac{1+2s}{10}\right) \sin (w(s))+ w'(s)\right].$$
\end{small}
Then 
\begin{small}
$$
\left| \mathscr G (t,s ,w_{1},w_{2})- \mathscr G(t,s ,v_{1},v_{2})\right|
\leq  \left\lbrace \left|w_{1}- v_{1}\right| + \left| w_{2}- v_{2}\right| \right\rbrace\,\,t,s\in [0,2]\,\, \text{and}\,\, w_{1},w_{2},v_{1},v_{2}\in \mathbb{R}.
 $$
\end{small}
Further, define $\mathscr F:[0,2] \times \mathbb{R} \times \mathbb{R}\times \mathbb{R} \to \mathbb{R} $ by
\begin{small}
\begin{align*}
&\mathscr F\left( t,  w(t),w'(t),\int_{0}^{t}\mathscr G\left( t, s, w(s),w'(s)\right) ds\right)\\
& =\frac{2}{10}- \frac{t^2}{1000}-\frac{(9-t)}{1000}+ \frac{\cos(w(t))}{100}- \frac{w'(t)}{100} +\frac{1}{100}  \int_{0}^{t} \mathscr G\left( t, s, w(s),w'(s)\right) ds.
\end{align*}
\end{small}
Then 
\begin{small}
\begin{align*}
& \left| \mathscr F (t,w_{1},w_{2},w_{3})- \mathscr F(t ,v_{1},v_{2},v_{3})\right|\\
 & \leq \frac{1}{100} \left\lbrace \left|w_{1}- v_{1}\right| + \left| w_{2}- v_{2}\right| +\left|w_{3}- v_{3}\right|\right\rbrace ,\,  t\in [0,2] \,\,\text{and}\,\, w_{1},w_{2},w_{3},v_{1},v_{2},v_{3}\in \mathbb{R}.
\end{align*}
\end{small}
All the conditions of the Theorem \ref{thm1} are satisfied with   $L_{ \mathscr F}=\frac{1}{100}$ and $L_{ \mathscr G}=1$. Take $ \gamma=2$  then we have 
\begin{small}
\begin{align*}
q= \frac{L_{\mathscr F}}{\gamma}\left(1+\frac{L_{\mathscr G}}{\gamma} \right)\left( 1+\left[ 1+\left\lbrace  T\beta+T\beta\left| \dfrac{\sum\limits_{k=1}^{p}c_{k}}{1+\sum\limits_{k=1}^{p}c_{k}} \right| \right\rbrace  \right]\frac{e^{\gamma\,T}}{\beta-1} \right)=0.8395<1.
\end{align*}
\end{small}
Thus the Theorem \ref{thm1} guarantee  uniqueness solutions of \eqref{10}--\eqref{12}. By actual substitution one can verify that $w(t)=\frac{(t+t^2)}{10},\, t\in [0,2]$ is the unique solution  \eqref{10}--\eqref{12}.
\section{Concluding Remarks}
We conclude the paper with the objectives achieved: the existence, uniqueness and dependency of solution for a second order Volterra integrodifferential equations with nonlocal and boundary conditions  through the Picard operators theory. Note that the fractional calculus is the branch of mathematical analysis, which has been on the rise over the decade, for its well-established theory and its important results that have resulted to the point it has attained today \cite{jose,jose1}. Although many papers have been published and the theory has been expanded, there are many avenues to be covered. In this sense, an interesting idea is to adapt some results obtained here for the field of fractional calculus, especially fractional differential equations, to put other results, and to present the improvements that can be obtained, in particular, making some comparisons through examples, in order to elucidate the results obtained and consequently, the differences between the whole case and the fractional one \cite{sousa,sousa1}.
\section*{Acknowledgement}
 The first author is financially supported by UGC, New Delhi, India (Ref: F1-17.1/2017-18/RGNF-2017-18-SC-MAH-43083).  The third author JVCS acknowledges the financial support of a PNPD-CAPES (process number nº88882.305834/2018-01) scholarship of the Postgraduate Program in Applied Mathematics of IMECC-Unicamp.

\end{document}